\documentclass[12pt]{amsart}
\usepackage{amssymb,amsmath,amsthm,slashed}
\newtheorem{theorem}{Theorem}
\newtheorem{lemma}[theorem]{Lemma}
\newtheorem{corollary}[theorem]{Corollary}
\newtheorem*{conjecture}{\bf Conjecture}

\theoremstyle{remark}

\newtheorem*{remark}{Remark}

\numberwithin{theorem}{section} \numberwithin{equation}{section}

\newcommand{\C}{\mathbb{C}}

\newcommand{\Z}{\mathbb{Z}}

\newcommand{\im}{\textnormal{Im}}

\begin{document}
\title[Moonshine and Donaldson invariants of $\C \mathrm{P}^2$]
{Moonshine and Donaldson invariants of $\C \mathrm{P}^2$}
\author{Andreas Malmendier and Ken Ono}
\address{Department of Mathematics, Colby College,
Waterville, Maine 04901}
\email{andreas.malmendier@colby.edu}
\address{Department of Mathematics and Computer Science, Emory University,
Atlanta, Georgia 30322} \email{ono@mathcs.emory.edu}
\thanks{The second author thanks the NSF and the Asa Griggs Candler Fund for their generous support.}
\begin{abstract}
Eguchi, Ooguri, and Tachikawa 
recently conjectured \cite{EOT} a new  {\it moonshine} phenomenon.
They conjecture that the coefficients of a certain mock modular form $H(\tau)$, which arises 
from the $K3$ surface elliptic genus, are sums of dimensions of irreducible representations
of the Mathieu group $M_{24}$.
We prove that $H(\tau)$ surprisingly also plays a significant role in the theory of Donaldson invariants.
We prove that the Moore-Witten~\cite{MW} $u$-plane integrals for $H(\tau)$
are the $\mathrm{SO}(3)$-Donaldson invariants of $\C \mathrm{P}^2$. 
This result then implies a moonshine phenomenon where
these invariants conjecturally are  expressions in the dimensions of the irreducible
representations of $M_{24}$. Indeed, we obtain an explicit expression for the Donaldson
invariant generating function $\mathrm{Z}(p,S)$ in terms of the derivatives of $H(\tau)$.
\end{abstract}
\maketitle

\section{Introduction and statement of results}

This paper concerns the deep properties of the modular forms and mock modular forms
which arise from a study of the $K3$ surface elliptic genus.
To define these objects, we require
Dedekind's eta-function $\eta(\tau):=q^{\frac{1}{24}}\prod_{n=1}^{\infty}(1-q^n)$
(note. $\tau\in \mathbb{H}$ throughout and $q:=e^{2\pi i \tau}$), and the
classical Jacobi theta function 
\begin{displaymath}
 \vartheta_{ab}(v|\tau) := \sum_{n\in\mathbb{Z}} q^{\frac{(2n+a)^2}{8}} \; e^{\pi i \,(2n+a)(v+\frac{b}{2})},
\end{displaymath}
where
$a,b \in \lbrace 0,1\rbrace$ and $v\in\mathbb{C}$.
We recall some standard identities. 

\medskip

\label{JacobiTheta}
\begin{displaymath}
 \begin{array}{|l|l|l|}
 \hline  && \\[-2ex]
 \vartheta_1(v|\tau) = \vartheta_{11}(v|\tau)
& \vartheta_1(0|\tau) = 0
& \vartheta_1'(0|\tau) = - 2\pi \eta^3(\tau) \\ [1ex]
\hline && \\[-2ex]
  \vartheta_2(v|\tau) = \vartheta_{10}(v|\tau) 
& \vartheta_2(0|\tau) = \sum_{n\in\mathbb{Z}} q^{\frac{(2n+1)^2}{8}} 
& \vartheta_2'(0|\tau) = 0 \\ [1ex]
\hline && \\[-2ex]
  \vartheta_3(v|\tau) = \vartheta_{00}(v|\tau)
& \vartheta_3(0|\tau) = \sum_{n\in\mathbb{Z}} q^{\frac{n^2}{2}}
& \vartheta_3'(0|\tau) = 0 \\ [1ex]
\hline && \\[-2ex]
  \vartheta_4(v|\tau) = \vartheta_{01}(v|\tau)
& \vartheta_4(0|\tau) = \sum_{n\in\mathbb{Z}} (-1)^n \, q^{\frac{n^2}{2}}
& \vartheta_4'(0|\tau) = 0\\  [1ex] \hline
\end{array}
\end{displaymath}
\medskip

\noindent
Moreover, for convenience we let $\vartheta_j(\tau):=\vartheta_j(0|\tau)$ for $j=2,3,4$.

\medskip
The $K3$ surface elliptic genus \cite{EST} is given by
\begin{displaymath}
Z(z|\tau)=8\left[\left(\frac{\vartheta_2(z|\tau)}{
\vartheta_2(\tau)}\right)^2+\left(\frac{\vartheta_3(z|\tau)}{
\vartheta_3(\tau)}\right)^2+\left(\frac{\vartheta_4(z|\tau)}{
\vartheta_4(\tau)}\right)^2\right]. 
\end{displaymath}
This expression is obtained by an
orbifold calculation on $T^4/\mathbb{Z}_2$ in \cite{EOTY}.
Its specializations at $z=0$, $z=1/2$ and $z=(\tau+1)/2$ 
gives the classical topological invariants
$\chi$=24,\,$\sigma$=16 and $\hat{A}=-2$ respectively.
Here we consider the following alternate representation obtained by
Eguchi and Hikami \cite{EguchiHikami} motivated by superconformal field theory:
\begin{displaymath}
Z(z|\tau)= \frac{\vartheta_1(z|\tau)^2}{\eta(\tau)^3} \, \Big(  24 \, \mu(z;\tau)  + H(\tau) \Big).
\end{displaymath}
Here $H(\tau)$ is defined by
\begin{equation} \label{H}
  H(\tau)  :=
   - 8 \hskip-3mm  \sum_{
    w \in \left\{
      \frac{1}{2} , \frac{1+\tau}{2} , \frac{\tau}{2}
    \right\}
  }\hskip-3mm 
  \mu(w; \tau)
  =
  2q^{-\frac{1}{8}}\left(
    - 1 + \sum_{n = 1}^\infty A_n \, q^n
  \right) ,
\end{equation}
where $\mu(z;\tau)$ is the famous function 
\begin{displaymath}
  \label{define_M}
  \mu(z;\tau)
  = 
  \frac{
    i \, e^{\pi i z} 
  }{
    \vartheta_{1}(z|\tau)}
  \sum_{n \in \mathbb{Z}}
  (-1)^n \,
  \frac{
    q^{\frac{1}{2} n(n+1)} \, e^{2 \pi i  n z}
  }{
    1 - q^n \, e^{2 \pi i z}
  } 
\end{displaymath}
defined by Zwegers \cite{Zwegers} in his thesis on Ramanujan's mock theta functions.

As explained in \cite{EguchiHikami}, $H(\tau)$ is the holomorphic part of a weight 1/2
harmonic Maass form, a so-called {\it mock modular form}.
Its first few coefficients $A_n$ are:
\smallskip

\begin{displaymath}
  \begin{array}{|c|rrrrrrrrr|cl}\hline
    n & 1 & 2 & 3 & 4 & 5 & 6 & 7 & 8 & \cdots   \\
    \hline
    A_n& 45 & 231 & 770 & 2277 & 5796 & 13915 &
    30843 & 65550 &\cdots \\ \hline
  \end{array}
\end{displaymath}
\smallskip

\noindent
Amazingly, Eguchi, Ooguri, and Tachikawa \cite{EOT} 
recognized these numbers as sums of dimensions of the irreducible representations of
the Mathieu group $M_{24}$. Indeed, the dimensions of the irreducible representations are (in increasing order):
\begin{displaymath}
\begin{split}
1, 23, 45, 45,\  &231, 231, 252, 253, 483, 770, 770, 990, 990, 1035, 1035, 1035,\\
&1265, 1771, 2024, 2277, 3312, 3520, 5313, 5544, 5796, 10395.
\end{split}
\end{displaymath}
One sees that $A_1, A_2, A_3, A_4$ and $A_5$ are dimensions, while
$$
A_6=3520+10395 \ \ \ \ \ {\text {\rm and}}\ \ \ \ \ A_7=10395+5796+5544+5313+2024+1771.
$$
We have their ``moonshine'' conjecture\footnote{This 
is analogous to the ``Monstrous Moonshine" conjecture by Conway and Norton
which related the coefficients of Klein's $j$-function
to the representations of the Monster \cite{Conway}. By work of Borcherds (and others)
\cite{Borcherds}, Monstrous Moonshine is now understood.} -- also referred to as ``umbral moonshine'' \cite{CDH}:
\begin{conjecture}[Moonshine]
The Fourier coefficients $A_n$ of $H(\tau)$ are given as special sums\footnote{As in the case of the Montrous Moonshine Conjecture,
there are many representations of the generic $A_n$, and so the proper formulation of this conjecture requires a precise description of
these sums \cite{GV}.}  of dimensions of irreducible representations of
the simple sporadic group $M_{24}$.
\end{conjecture}

Here we prove that the coefficients of $H(\tau)$ encode further deep information. 
We compute the numbers $\pmb{\mathrm{D}}_{m,2n}[H(\tau)]$, the Moore-Witten \cite{MW}  $u$-plane integrals for $H(\tau)$,
and we prove that they are, up to a multiplicative factor of 12, the
$SO(3)$-Donaldson invariants for $\C \mathrm{P}^2$.
These invariants are a sequence of rational numbers which together form a diffeomorphism class invariant for
$\C \mathrm{P}^2$  (for background see \cite{DonaldsonKronheimer, Goettsche, GoettscheZagier, MO}).

\begin{theorem}\label{mainthm}
For all $m,n \in \mathbb{N}_0$, the
$\mathrm{SO}(3)$-Donaldson invariants $\pmb{\Phi}_{m,2n}$ for $\mathbb{C}\mathrm{P}^2$
satisfy
\begin{displaymath}
 12 \,  \pmb{\Phi}_{m,2n} = \pmb{\mathrm{D}}_{m,2n}[H(\tau)] \;.
\end{displaymath}
\end{theorem}

\begin{remark} The $u$-plane integrals $\pmb{\mathrm{D}}_{m,2n}[H(\tau)]$ are given explicitly in terms of
the coefficients of $H(\tau)$ (see \ref{uplane0}). Therefore, 
the Eguchi-Ooguri-Tachikawa Moonshine
Conjecture implies that these Donaldson invariants are given explicitly in terms of the dimensions of
the irreducible representations of $M_{24}$. We will discuss the numerical identities implied by Theorem \ref{mainthm}
in Section \ref{sec4}. We also describe the Donaldson invariant generating function in terms of derivatives
of $H(\tau)$.
\end{remark}

This paper builds upon earlier work by the authors \cite{MO} on the Moore-Witten Conjecture for $\C \mathrm{P}^2$. We shall
make substantial
use of the results in that paper, and 
we will recall the main facts that we need to prove Theorem~\ref{mainthm}.

In Section~\ref{mock} we recall basic facts about those weight 1/2
harmonic Maass forms whose shadow is the cube of Dedekind's eta-function.
In Section~\ref{donaldsonsection} we recall and apply the main results from \cite{MO}.
In particular, we recall the relationship between the
$u$-plane integrals for such forms and the $SO(3)$-Donaldson invariants for $\C \mathrm{P}^2$. We then
 conclude with the proof of Theorem~\ref{mainthm}.

\section{Certain harmonic Maass forms}\label{mock}

We let
$M(\tau)$ be a  weight 1/2 harmonic Maass form\footnote{These forms were first defined by
Bruinier and Funke \cite{BF} in their work on geometric theta lifts.}
(for definitions see \cite{BF, OnoCDM, ZagierBourbaki}) for $\Gamma(2)\cap \Gamma_0(4)$
 whose
shadow\footnote{The term {\it shadow} was coined by Zagier
in \cite{ZagierBourbaki}.}  is $\eta(\tau)^3$.
Namely, we have that
\begin{equation}\label{DE}
 \sqrt{2} i \; \frac{d}{d\bar{\tau}} \; M\left(\tau\right)= \frac{1}{\sqrt{\im \tau}} \, \overline{\eta^3(\tau)} \;.
\end{equation}
For such $M(\tau)$, we write $M(\tau) = M^+(\tau) + M^-(\tau)$, where the {\it holomorphic part},   a {\it mock modular form}, is
$M^+(\tau)=q^{-1/8} \, \sum_{n\ge0} H_n \, q^{n/2}$.
The {\it non-holomorphic part} $M^{-}(\tau)$ is
\begin{displaymath}
\begin{split}
M^{-}\left(\tau\right)  =
- \frac{2i}{\sqrt{\pi}} \; \sum_{l \ge 0} (-1)^l \; \Gamma\left( \frac{1}{2} , \pi \, \frac{(2l+1)^2}{2} \, \im\tau \right) \; q^{-\frac{(2l+1)^2}{8}},
\end{split}
\end{displaymath}
where $\Gamma(1/2,t)$ is the incomplete Gamma function. This follows from Jacobi's identity
$$
\eta(\tau)^3=q^{\frac{1}{8}}\sum_{n=0}^{\infty}(-1)^n(2n+1)q^{\frac{n^2+n}{2}}.
$$

\begin{remark}
Note that the non-holomorphic part $M^{-}(\tau)$ is the same for
every weight 1/2 harmonic Maass form with shadow $\eta^3(\tau)$ since
this part is
obtained as the ``Eichler-Zagier'' integral of the shadow.
However, the holomorphic part is not uniquely determined.
It is unique up to the addition of a
{\it weakly holomorphic modular form}, a form whose poles (if any) are supported at cusps. 
\end{remark}

The next result gives families of modular forms from such an $M(\tau)$ using Cohen brackets. To make this precise, we recall the two Eisenstein
series
\begin{displaymath}
E_2(\tau):=1-24\sum_{n=1}^{\infty}\sum_{d\mid n}dq^n \ \ \ {\text {\rm and}}\ \ \
\widehat{E}_2(\tau):=E_2(\tau)-\frac{3}{\pi \im \tau}.
\end{displaymath}

The authors proved the following lemma in \cite{MO}.
\begin{lemma}\label{EkQ}{\text {\rm [Lemma 4.10 of \cite{MO}]}}
Assuming the hypotheses above, we have that
\begin{equation*}
 \mathcal{E}^k_{\frac{1}{2}} \left[ M(\tau) \right] := \sum_{j=0}^k (-1)^j \; \binom{k}{j} \; \frac{\Gamma\left(\frac{1}{2}\right)}{\Gamma\left(\frac{1}{2}+j\right)} \; 2^{2j} \; 3^j
 \; E_2^{k-j}(\tau) \; \left(q \, \frac{d}{dq} \right)^j M\left(\tau\right)
\end{equation*}
is modular of weight $2k+1/2$ for $\Gamma(2)\cap \Gamma_0(4)$, and it satisfies
\begin{displaymath}
 \sqrt{2} i \; \frac{d}{d\bar{\tau}}\; \mathcal{E}^k_{\frac{1}{2}} \left[ M(\tau) \right]  = \frac{1}{\sqrt{\im \tau}} \, \widehat{E}_2^k(\tau) \; \overline{\eta^3(\tau)} \;.
\end{displaymath}
\end{lemma}

This lemma implies the following corollary:
\begin{corollary}\label{CorDiff}
If  $M(\tau)$ and $\widetilde{M}(\tau)$ are weight $\frac{1}{2}$ harmonic Maass forms on $\Gamma(2)\cap \Gamma_0(4)$ whose
shadow is $\eta(\tau)^3$, then
\begin{equation*}
 \mathcal{E}^k_{\frac{1}{2}} \left[ M(\tau)\right] -  \mathcal{E}^k_{\frac{1}{2}} 
\left[ \widetilde{M}(\tau) \right] =  \mathcal{E}^k_{\frac{1}{2}} \left[ M(\tau) - \widetilde{M}(\tau) \right]  =  
  \mathcal{E}^k_{\frac{1}{2}} \left[ M^+(\tau) - \widetilde{M}^+(\tau) \right] 
\end{equation*}
is a weakly holomorphic modular form of weight  $2k+1/2$.
\end{corollary}

\subsection{The $\mathcal{Q}(q)$ series}
Here we recall one explicit example of a harmonic Maass form which plays the role
of $M(\tau)$ in the previous subsection.
To this end, we
define modular forms $A(\tau)$ and $B(\tau)$ by
\begin{displaymath}
\label{B-functions}
\begin{split}
\mathcal{A}(\tau):=A(8\tau)&=\sum_{n=-1}^{\infty}a(n)q^n:=\frac{\eta(4\tau)^8}{\eta(8\tau)^7}=
q^{-1}-8q^3+27q^7-\cdots,\\
\mathcal{B}(\tau):=B(8\tau)&=\sum_{n=-1}^{\infty}b(n)q^n:=\frac{\eta(8\tau)^5}{\eta(16\tau)^4}=
q^{-1}-5q^7+9q^{15}-\cdots.
\end{split}
\end{displaymath}
We sieve on the Fourier expansion of $\mathcal{A}(\tau)$ to define the modular
forms
\begin{displaymath}
\label{A-functions}
\begin{split}
\mathcal{A}_{3,8}(\tau)&:=A_{3,8}(8\tau)=\sum_{n\equiv 3\pmod 8}a(n)q^n=-8q^3-56q^{11}+\cdots,\\
\mathcal{A}_{7,8}(\tau)&:=A_{7,8}(8\tau)=\sum_{n\equiv 7\pmod 8}a(n)q^n=q^{-1}+27q^7+105q^{15}+\cdots.
\end{split}
\end{displaymath}
We also recall the definition of the following mock theta function
\begin{displaymath}
\label{M-function}
\mathcal{M}(q):=q^{-1}\sum_{n=0}^{\infty}\frac{(-1)^{n+1}q^{8(n+1)^2}\prod_{k=1}^n
(1-q^{16k-8})}{\prod_{k=1}^{n+1}(1+q^{16k-8})^2}=-q^7+2q^{15}-3q^{23}+\cdots,
\end{displaymath}
We define
\begin{displaymath}
\mathcal{Q}^+(q)=\mathcal{Q}^+(\tau):=-\frac{7}{2}\mathcal{A}_{3,8}(\tau)+\frac{3}{2}\mathcal{A}_{7,8}(\tau)-
\frac{1}{2}\mathcal{B}(\tau)+4\mathcal{M}(q)\;,
\end{displaymath}
and so we have that
\begin{equation}\label{Qplus}
 \mathcal{Q}^+\left(\tau/8\right) = \frac{1}{q^{\frac{1}{8}}}\left( 1 + 28 \, q^{\frac{1}{2}} + 39 \, q + 196 \, q^{\frac{3}{2}} + 161 \, q^2 + \dots \right) \;.
\end{equation}
In terms of this $q$-series, the authors proved the following theorem in \cite{MO}.
\begin{theorem}\label{QasMu}{\text {\rm [Theorem 7.2 of \cite{MO}]}}
The function $\mathcal{Q}^+(\tau/8)$ is the holomorphic part of a weight 1/2
harmonic  Maass form on $\Gamma(2)\cap \Gamma_0(4)$ whose
shadow is $\eta(\tau)^3$.
\end{theorem}

\section{$u$-plane integrals,
Donaldson invariants and the proof of Theorem~\ref{mainthm}}\label{donaldsonsection}

Suppose again that $M(\tau)$ is a weight 1/2 harmonic Maass form on $\Gamma(2)\cap \Gamma_0(4)$
whose shadow is $\eta(\tau)^3$.
For $m,n \in \mathbb{N}_0$, the authors proved that the quantities 
\begin{equation}
\label{uplane0}
   \pmb{\mathrm{D}}_{m,2n}[M^+(\tau)] := \sum_{k=0}^n \frac{(-1)^{k+1}}{2^{n-1} \; 3^{n}} \;
 \frac{(2n)!}{(n-k)! \; k!} \;  \left[ \, \frac{\vartheta^9_4(\tau) \,
\left[\vartheta^4_2(\tau) + \vartheta^4_3(\tau)\right]^{m+n-k}}{\left[\vartheta_2(\tau) \, \vartheta_3(\tau)\right]^{2m+2n+3}} \;
  \mathcal{E}^k_{\frac{1}{2}} \left[ M^+(\tau)\right]  \right] _{q^0} \;,
\end{equation}
where $[.]_{q^0}$ denotes the constant coefficient term, 
are the Moore-Witten {\it $u$-plane integrals  for
} $M(\tau)$ (cf. \cite{MO}).
Notice that if $\widetilde{M}(\tau)$ is another such form, then
\begin{equation}
  \pmb{\mathrm{D}}_{m,2n}[M^+(\tau)]  -  \pmb{\mathrm{D}}_{m,2n}[\widetilde{M}^+(\tau)]  =  \pmb{\mathrm{D}}_{m,2n}[M^+(\tau) - \widetilde{M}^+(\tau)] \;. 
\end{equation}

In their seminal paper \cite{MW}, Moore and Witten essentially conjectured that the $u$-plane integrals
in (\ref{uplane0}) for a suitable $M^+(\tau)$ should give the
$\mathrm{SO}(3)$-Donaldson invariants of $\mathbb{C}\mathrm{P}^2$. These invariants  are an 
infinite sequence of rational numbers $\pmb{\mathrm{\Phi}}_{m,2n}$ labeled by integers $m,n \in \mathbb{N}$
that can be assembled in a generating function in the two formal variables $p,S$:
\begin{displaymath}
\label{generatingfunction}
 \pmb{\mathrm{Z}}(p,S) =  \sum_{m ,n \ge 0} \pmb{\mathrm{\Phi}}_{m, 2n} \;
 \; \frac{p^{m}}{m!}\frac{S^{2n}}{(2n)!}.
\end{displaymath}
This power series is a diffeomorphism invariant for $\C \mathrm{P}^2$.
The main theorem in \cite{MO} proved this conjecture for $\mathcal{Q}^+(\tau/8)$.
\begin{theorem} \label{koeffizienten}{\text {\rm [Theorem 1.1 of \cite{MO}]}}
For $m,n \in \mathbb{N}_0$ we have that
\begin{displaymath}
\label{uplane0b}
 \pmb{\mathrm{\Phi}} _{m,2n}
=   \pmb{\mathrm{D}}_{m,2n}[\mathcal{Q}^+(\tau/8)] \;.
\end{displaymath}
\end{theorem}

\medskip
Using the work in \cite{MO}, we prove the following important theorem.
\begin{theorem} \label{kernel}
Let $M(\tau)$ be  as above, then for all $m,n \in \mathbb{N}_0$ we have:
\begin{displaymath}
\begin{split}
\pmb{\mathrm{D}}_{m,2n}[\mathcal{Q}^+(\tau/8)] - \pmb{\mathrm{D}}_{m,2n}[{M}^+(\tau)]  = \pmb{\mathrm{D}}_{m,2n}[\mathcal{Q}^+(\tau/8) - {M}^+(\tau)] 
  =0 \;.
\end{split}
\end{displaymath}
\end{theorem}
\begin{proof}
We  prove that the constant terms vanish in expressions of the form
\begin{displaymath}
\label{difference}
\begin{split}
 \sum_{k=0}^n \frac{(-1)^{k+1}}{2^{n-1} \; 3^{n}} \;
 \frac{(2n)!}{(n-k)! \; k!} \;  \frac{\vartheta^9_4(\tau) \,
\left[\vartheta^4_2(\tau) + \vartheta^4_3(\tau)\right]^{m+n-k}}{\left[\vartheta_2(\tau) \, \vartheta_3(\tau)\right]^{2m+2n+3}} \;
  \mathcal{E}^k_{\frac{1}{2}} \left[ \mathcal{Q}^+(\tau/8) - {M}(\tau) \right]  \;. 
  \end{split}
\end{displaymath}
It is sufficient to show that this is the case for each summand. Therefore, after rescaling
$\tau \to 8\tau$ and $q \to q^8$ it is enough to show that the constant vanishes in
\begin{equation}
\label{difference2}
\begin{split}
&  \frac{\Theta^9_4(\tau) \,
 \left[16 \, \Theta^4_2(\tau) + \Theta^4_3(\tau)\right]^{m+n-k}}{\left[\Theta_2(\tau) \, \Theta_3(\tau)\right]^{2m+2n+3}} \;
  \mathcal{E}^k_{\frac{1}{2}} \left[ \mathcal{Q}^+(\tau) - {M}(8\tau) \right] \\
= & \frac{\Theta^9_4(\tau)}{\Theta_2(\tau)\Theta_3(\tau)\eta(8\tau)^3} \;
 \frac{ \left[16 \, \Theta^4_2(\tau) + \Theta^4_3(\tau)\right]^{m+n-k}}{\left[\Theta_2(\tau) \, \Theta_3(\tau)\right]^{2m+2n-2k}} \;
   \frac{\eta(8\tau)^3}{(\Theta_2(\tau)\Theta_3(\tau))^{2k+2}}\;  \mathcal{E}^k_{\frac{1}{2}} \left[ \mathcal{Q}^+(\tau) - {M(8\tau)}
 \right]. 
  \end{split}
\end{equation}
Here the classical theta functions are defined by
\begin{displaymath}\label{theta2}
\Theta_2(\tau):=\frac{\eta(16\tau)^2}{\eta(8\tau)}=\sum_{n=0}^{\infty}q^{(2n+1)^2}=
q+q^9+q^{25}+\cdots,
\end{displaymath}
\begin{displaymath}\label{theta3}
\Theta_3(\tau):=\frac{\eta(8\tau)^5}{\eta(4\tau)^2\eta(16\tau)^2}=1+2\sum_{n=1}^{\infty}
q^{4n^2}=1+2q^4+2q^{16}+2q^{36}+\cdots,
\end{displaymath}
\begin{displaymath}\label{theta4}
\Theta_4(\tau):=\frac{\eta(4\tau)^2}{\eta(8\tau)}=1+2\sum_{n=1}^{\infty}(-1)^nq^{4n^2}=
1-2q^4+2q^{16}-2q^{36}+\cdots.
\end{displaymath}
These are related to the theta functions $\vartheta_2(\tau), \vartheta_3(\tau)$ and $\vartheta_4(\tau)$ by
$$\vartheta_2(\tau)=2\Theta_2\left(\frac{\tau}{8}\right),\ \ \
\vartheta_3(\tau)=\Theta_3\left(\frac{\tau}{8}\right),\ \ \
\vartheta_4(\tau)=\Theta_4\left(\frac{\tau}{8}\right).
$$

We define  a weakly holomorphic modular function  by
\begin{eqnarray}\label{hatZ0}
\widehat{Z_0}(q)=\widehat{Z_0}(\tau) := \frac{E^*(4\tau)}{\Theta_2(\tau)^2\Theta_3(\tau)^2}
\end{eqnarray}
where $E^*(4\tau)$ is the weight 2 Eisenstein series with
 \begin{displaymath}
 E^*(4\tau) =16\Theta_2(\tau)^4+\Theta_3(\tau)^4=1+24q^4+24q^2+\cdots \;.
\end{displaymath} 
A calculation shows that
\begin{displaymath}
q\frac{d}{dq}\widehat{Z_0}(q) = \frac{\Theta_4(\tau)^9}{\Theta_2(\tau)\Theta_3(\tau)\eta(8\tau)^3}.
\end{displaymath}
Equation (\ref{difference2}) becomes
\begin{equation}
\label{difference3}
 q\frac{d}{dq}\widehat{Z_0}(q) \cdot  \widehat{Z_0}(q)^{m+n-k}\cdot \mathcal{H}_{k}(q)\;.
 \end{equation}
where
\begin{equation}
\mathcal H _{k}(q):= \frac{\eta(8\tau)^3}{\left(\Theta_2(\tau)\Theta_3(\tau)\right)^{2k+2}} \;  \mathcal{E}^k_{\frac{1}{2}} \left[ \mathcal{Q}^+(\tau) - {M}(8\tau) \right] \;.
\end{equation}
 To prove the theorem, it suffices to show that the constant term in (\ref{difference3}) vanishes. Hence, it is enough to show that $\mathcal H _{k}(q)$ is a polynomial in $\widehat{Z_0}(q)$. To this end, we define $M_0^*(\Gamma_0(8))$ to be the space of modular function on $\Gamma_0(8)$ which are holomorphic 
 away from infinity, and is a subspace of $\C((q^2))$. One can easily  verify that $M_0^*(\Gamma_0(8))$ is precisely the set of polynomials in $\widehat{Z_0}(q)$.
From Corollary~\ref{CorDiff} we can observe that $\mathcal H _{k}(q)$ is modular with weight $0$.  A calculation shows that  $\left(\Theta_2(\tau)\Theta_3(\tau)\right)^{-2}=
q^{-2} \, f(q^4)$ is holomorphic away from infinity, and $f(q)\in \Z[[q]]$. We also have $\eta(8\tau)^3 = q \, g(q^8)$ and $\mathcal{E}^k_{\frac{1}{2}} \left[ \mathcal{Q}^+(\tau) - {M}(8\tau) \right]  = q^{-1} \, h(q^4)$, where $g(q), h(q)\in \Z[[q]]$. Hence, $\mathcal{H}_k(q) \in \mathbb{C}((q^2))$ is modular of
 weight 0 on $M_0^*(\Gamma_0(8))$, and so is a polynomial in $\widehat{Z}_0(\tau)$.
\end{proof}

\subsection{Proof of Theorem~\ref{mainthm}}
Since $H(\tau)$ is the mock modular part
of a  weight 1/2 harmonic Maass form on $\Gamma(2)\cap\Gamma_0(4)$ whose
shadow is the $8 \cdot 3 \cdot \eta(\tau)^3/2 = 12 \, \eta(\tau)^3$, it follows from Theorem \ref{koeffizienten} and \ref{kernel}
that for $m,n \in \mathbb{N}_0$ we have:
\begin{displaymath}
\begin{split}
\pmb{\Phi}_{m,2n} & = \pmb{\mathrm{D}}_{m,2n}[\mathcal{Q}^+(\tau/8)]  
= \pmb{\mathrm{D}}_{m,2n}[H(\tau)/12] +  \pmb{\mathrm{D}}_{m,2n}\Big[\mathcal{Q}^+(\tau/8) - H(\tau)/12\Big] \\
& = \pmb{\mathrm{D}}_{m,2n}[H(\tau)/12] 
 = \frac{1}{12} \, \pmb{\mathrm{D}}_{m,2n}[H(\tau)].
\end{split}
\end{displaymath}


\subsection{Discussion of the identities implied by Theorem \ref{mainthm}}
\label{sec4}

In the table below we list the first non-vanishing $\mathrm{SO}(3)$-Donaldson invariants $\pmb{\Phi}_{m,2n}$
of $\C\mathrm{P}^2$ as well as the coefficients $\pmb{\mathrm{D}}_{m,2n}[M^+(\tau)]$ when the mock modular form
is given as $M^+(\tau)=q^{-1/8} \, \sum_{k\ge0} H_k \, q^{k/2}$. In general, $\pmb{\mathrm{D}}_{m,2n}[M^+(\tau)]$ is 
nonvanishing for $m+n\equiv 0\pmod{2}$ and a rational linear combination of the first $(m+n)/2+1$ coefficients of $M^+(\tau)$.

\medskip
\begin{center}
\begin{tabular}{|c||l|l|}
\hline
\raisebox{-0.5ex}{$(m,n)$} & \raisebox{-0.5ex}{$\pmb{\Phi}_{m,2n}$} & \raisebox{-0.5ex}{$\pmb{\mathrm{D}}_{m,2n}[M^+(\tau)]$}\\ [1ex]
\hline
\raisebox{-0.5ex}{$(0,0)$}   &  \raisebox{-0.5ex}{$-1$}                  &  \raisebox{-0.5ex}{$-\frac{1}{4} H_1 + 6 H_0$} \\ [1ex]
\hline
 \raisebox{-0.5ex}{$(0,2)$} &   \raisebox{-0.5ex}{$- \frac{3}{16}   $}   &  \raisebox{-0.5ex}{$- \frac{49}{64} H_2 +  \frac{9}{4}  H_1 -  \frac{2133}{64}  H_0$} \\ [1ex]
 $(1,1)$ &  $- \frac{5}{16}   $    & $ - \frac{7}{64} H_2 +  \frac{1}{4}  H_1 -   \frac{195}{64}  H_0$ \\ [0.5ex]
 $(2,0)$ &  $- \frac{19}{16}  $    & $  -\frac{1}{64} H_2 -  \frac{1}{4}  H_1 +   \frac{411}{64}  H_0$ \\ [0.5ex]
 \hline
 \raisebox{-0.5ex}{$(0,4)$}&  \raisebox{-0.5ex}{$ - \frac{232}{256}   $ } &
\raisebox{-0.5ex}{$- \frac{14641}{1024} H_3 + \frac{2401}{128}  H_2 +  \frac{44631}{1024}  H_1 + \frac{108741}{128} H_0$}\\ [1ex]
$(1,3)$&$ - \frac{152}{256}   $  & $- \frac{1331}{1024}  H_3 -    \frac{49}{128} H_2 +  \frac{10341}{1024}  H_1 - \frac{1749}{128}   H_0$\\ [0.5ex]
$(2,2)$&$ - \frac{136}{256}   $  & $- \frac{121}{1024}   H_3 -    \frac{91}{128} H_2 +  \frac{2895}{1024}   H_1 - \frac{3687}{128}   H_0$\\ [0.5ex]
$(3,1)$&$ - \frac{184}{256}   $  & $- \frac{11}{1024}    H_3 - \frac{29}{128}    H_2 +  \frac{589}{1024}    H_1 - \frac{753}{128}    H_0$\\ [0.5ex]
$(4,0)$&$ - \frac{680}{256}   $  & $- \frac{1}{1024}     H_3 - \frac{7}{128}     H_2 -  \frac{505}{1024}    H_1 + \frac{1725}{128}   H_0$\\ [0.5ex]
\hline
\end{tabular}
\end{center}

\medskip

Theorem \ref{koeffizienten} states that choosing $M^+(\tau)=\mathcal{Q}^+(\tau/8)$ from (\ref{Qplus}) 
we find equality of the Donaldson invariants $\pmb{\Phi}_{m,2n}$ and the $u$-plane integral $\pmb{\mathrm{D}}_{m,2n}[M^+(\tau)]$. In fact, setting 
$H_0=1, H_1=28, H_2=39, H_3=196$ in the third column of the table above gives the Donaldson invariants of the second column.

On the other hand, the choice $M^+(\tau)=H(\tau)/12$ from  (\ref{H}) implies that $H_0 = -1/6$, $H_{2k}=A_k/6$, $H_{2k+1}=0$ for $k \in \mathbb{N}$.
Theorem \ref{mainthm} states that choosing $M^+(\tau)=H(\tau)/12$ we still find equality of
the Donaldson invariants $\pmb{\Phi}_{m,2n}$ and the $u$-plane integral $\pmb{\mathrm{D}}_{m,2n}[M^+(\tau)]$.
 In fact, setting  $H_0=-1/6, H_1=0, H_2=45/6, H_3=0$ in the third column of the table gives the Donaldson invariants of the second column as well.

The proof of Theorem \ref{mainthm} implies the following form for the generating function $\pmb{\mathrm{Z}}(p,S)$ of the 
$\mathrm{SO}(3)$-Donaldson invariants of $\mathbb{C}\mathrm{P}^2$ in terms of the mock modular form $H(\tau)$:
\begin{equation}
\begin{split}
  \pmb{\mathrm{Z}}(p,s) =  -  \sum_{m,n \ge 0} & \; \frac{p^m \, S^{2n}}{2^{2m+3n+4} \cdot 3^{n+1} \cdot  m! \cdot n!} \\
\times &   \left[ \, q\frac{d}{dq}\widehat{Z_0}(q) \; \sum_{k=0}^n (-1)^k \,  \binom{n}{k} \, \widehat{Z_0}(q)^{m+n-k}\; \widehat{\mathcal{E}^{k}}[H(8\tau)] \right]_{q^0}  \;,
\end{split}
\end{equation}
where $\widehat{Z_0}(q)$ was defined in (\ref{hatZ0}) and we have set
\begin{equation*}
 \widehat{\mathcal{E}^{k}}[H(8\tau)] = \frac{\eta(8\tau)^3}{\left(\Theta_2(\tau)\Theta_3(\tau)\right)^{2k+2}} \;  \mathcal{E}^k_{\frac{1}{2}} \left[ {H}(8\tau) \right] \;.
\end{equation*}


\begin{thebibliography}{GKZ}

\bibitem{Conway}
\emph{J. H. Conway and S. P. Norton}, Monstrous Moonshine,
Bull. London Math. Soc. \textbf{11} (1979), 308-339.

\bibitem{Borcherds}
\emph{R. E. Borcherds}, Monstrous Moonshine and Monstrous Lie superalgebras,
Invent. Math. \textbf{109} (1992), 405-444.

\bibitem{BF} {\em J. H. Bruinier and J. Funke},
On two geometric theta lifts, Duke Math. J. {\bf 125} (2004), 
45-90.

\bibitem{CDH}
{\em M. C. N. Cheng, J. F. R. Duncan, J. A. Harvey},
Umbral Moonshine,
 arXiv:1204.2779v2 [math.RT].

\bibitem{DonaldsonKronheimer}
\emph{S. K. Donaldson and P. B. Kronheimer},
The geometry of four-manifolds,
Oxford Mathematical Monographs, Oxford Science Publications, The Clarendon Press, Oxford University Press, New York, 1990.


\bibitem{EOTY}
\emph{T.~Eguchi, H.~Ooguri, A.~Taormina and S.~K.~Yang},
Superconformal Algebras And String Compactification On Manifolds With SU(N) Holonomy,'
 Nuclear Phys. B \textbf{315} (1989), no. 1, 193-221. 

\bibitem{EST}
\emph{T.~Eguchi, Y.~Sugawara, A.~Taormina},
Modular forms and elliptic genera for ALE spaces. 
Adv. Stud. Pure Math., \textbf{61}, Math. Soc. Japan, Tokyo, 2011.

\bibitem{EguchiHikami}
\emph{T.~Eguchi, K.~Hikami},
Superconformal algebras and mock theta functions. II. Rademacher expansion for K3 surface.
Commun. Number Theory Phys. \textbf{3} (2009), no. 3, 531-554. 

\bibitem{EOT}
\emph{T.~Eguchi, H.~Ooguri, Y.~Tachikawa},
Notes on the K3 surface and the Mathieu group M24. 
Exp. Math. \textbf{20} (2011), no. 1, 91-96. 

\bibitem{GV}
{\em M. R. Gaberdiel, R. Volpato},
Mathieu Moonshine and Orbifold K3s,
arXiv:1206.5143v1 [hep-th].

\bibitem{Goettsche}
\emph{L. G\"ottsche},
Modular forms and Donaldson invariants for $4$-manifolds with $b\sb {+}=1$,
J. Amer. Math. Soc. \textbf{9} (1996), no. 3, 827-843.

\bibitem{GoettscheZagier}
\emph{L. G\"ottsche and D. Zagier}, Jacobi forms and the structure of Donaldson invariants for $4$-manifolds with $b\sb +=1$,
Selecta Math. \textbf{4} (1998), no. 1, 69-115.

\bibitem{MO}
\emph{A.~Malmendier, K.~Ono},
$\mathrm{SO}(3)$-Donaldson invariants of $\mathbb{C}\mathrm{P}^2$ and mock theta functions.
Geometry and Topology. accepted for publication. arXiv:0808.1442 [math.DG].

\bibitem{MW}
\emph{G. Moore and E. Witten},
Integration over the $u$-plane in Donaldson theory,
Adv. Theor. Math. Phys.  \textbf{1}  (1997),  no. 2, 298-387.

\bibitem{OnoCDM} \emph{K. Ono},
Unearthing the visions of a master: Harmonic Maass forms and number
theory, Proceedings of the 2008 Harvard-MIT Current Developments in Mathematics
Conference, Int. Press, Somerville, Ma., 2009, 347-454.

\bibitem{ZagierBourbaki} \emph{D. Zagier}, Ramanujan's mock theta functions
and their applications [d'apr\`es Zwegers and Bringmann-Ono], S\'em. Bourbaki
(2007/2008), Ast\'erisque, No. 326, Exp No. 986, vii-viii, (2010) 143-164.
\bibitem{Zwegers} \emph{S. P. Zwegers}, {mock theta functions},
Ph.D. Thesis, Universiteit Utrecht, 2002.
\end{thebibliography}
\end{document}